\documentclass[12pt]{amsart}
\usepackage[a4paper,margin=1.9cm,top=2.5cm,bottom=2.5cm,centering,vcentering]{geometry}

\usepackage{concmath}

\theoremstyle{plain}
\newtheorem{theorem}{Theorem}

\theoremstyle{definition}
\newtheorem{definition}{Definition}
\usepackage{graphicx}
\usepackage{xcolor}
\usepackage{enumerate}
\usepackage[shortlabels]{enumitem}

\usepackage{parskip}
\usepackage{tikz}
\usetikzlibrary{shapes.geometric, arrows, shadows, calc,positioning}
\usetikzlibrary{graphs, graphs.standard, quotes, arrows.meta}

\usepackage[bookmarksnumbered, plainpages,colorlinks=true,allcolors=blue]{hyperref}

\usepackage{tkz-graph}

\theoremstyle{remark}

\newtheorem{example}{Example}

\newcommand\set [1]{\{#1\}}
\newcommand{\tpmod}[1]{{\@displayfalse\pmod{#1}}}
\usepackage{amsmath}
\numberwithin{equation}{section}
\usepackage{tikz}
\newcommand{\Cross}{\mathbin{\tikz [x=1.4ex,y=1.4ex,line width=.2ex] \draw (0,0) -- (1,1) (0,1) -- (1,0);}}
\allowdisplaybreaks[4]

\begin{document}
\title[Signed Roman Domination Number]{The Signed Roman Domination Number of Ladder graphs, circular Ladder graphs and their complements}

\author{Dilbak Haje}
\author{Delbrin Ahmed}
\address{University of Duhok, University campus, Zakho street, Duhok, Kurdistan region, Iraq} 
\email{dilpakhaje@uod.ac, delbrin.ahmed@uod.ac}

\author{Hassan Izanloo}
\address{School of Mathematics, University of Leeds, Leeds, LS2 9JT, UK}
\email{h.izanloo@leeds.ac.uk}

\author{Manjil Saikia}
\address{Mathematical and Physical Sciences division, School of Arts and Sciences, Ahmedabad University, Navrangpura, Ahmedabad - 380009, Gujarat, India
}
\email{manjil@saikia.in}

\begin{abstract}
Let $G=(V,E)$ be a finite connected simple graph with vertex set $V$ and edge set $E$. A signed Roman dominating function (SRDF) on a graph $G$ is a function  $f: V \rightarrow \{-1, 1, 2\}$ that satisfies two conditions: (i) $\sum_{y\in N[x]} f(y)\geq1$ for each $x\in V$, where the set $N[x]$ is the closed neighborhood of $x$ consisting of $x$ and vertices of $V$ that are adjacent to $x$, and (ii) each vertex $x\in V$  where $f(x) = -1$ is adjacent to at least one vertex  $y\in V$ where $f(y)=2$. The weight of a SRDF is the sum of its function values over all vertices. The signed Roman domination number of  $G$, denoted by $\gamma_{SR}(G)$, is the minimum weight of a SRDF on $G$. In this paper, we investigate the signed Roman domination number of the Ladder graph $LG_n$, the circular Ladder graph $CL_n$ and their complements.
\end{abstract}



\maketitle

\section{Introduction}

Let $G=(V(G), E(G))$ be a connected graph of order $n=|V(G)|$ and size $m=|E(G)|$. In this paper, only (non trivial) simple graphs, i.e., finite, undirected graphs without loops or multiple edges are considered. When $u$ is a vertex of $G$, then the open neighbourhood of $u$ in $G$ is the set $N_G(u)=\set{v:\set{u,v}\in E(G)}$ and the closed neighbourhood of $u$ in $G$ is the set $N_G[u]= N_G(u) \cup \set{u}$. The degree of vertex $u$ is the number of edges adjacent to $u$ and is denoted by $deg_G(u)$. A graph is said to be regular if all of its vertices have the same degree. The minimum degree and the maximum degree of $G$ are denoted by $\delta(G)$ and $\Delta(G)$, respectively. A graph is called $k$-regular if each vertex of the graph has degree $k$. We write $K_n$ for the complete graph of order $n$, $P_n$ for the path graph of order $n$ and $C_n$ for a cycle of length $n$. The complement of a graph $G=(V, E)$ is a graph $G^C=(V,E^C)$ with the same vertex set $V$ with two vertices $uv\in E^C$ if and only if $\{u,v\}\notin E$, for all pairs $u\neq v\in V$. 

A set $D \subseteq V(G)$ is called a \textit{dominating set} of $G$ if each vertex outside $D$ has at least one neighbour in $D$. The minimum cardinality of a dominating set of $G$ is the domination number of $G$ and is denoted by $\gamma(G)$. For example, the domination numbers of the $n$-vertex complete graph, path, and cycle are given by $\gamma(K_n)=1 , \gamma(P_n)=\left\lceil\dfrac{n}{3}\right\rceil \text{ and } \gamma(C_n)=\left\lceil\dfrac{n}{3}\right\rceil$, respectively (see \cite{haynes2017domination}). Domination is a rapidly developing area of research in graph theory, and has various applications to several other practical areas. The concept of domination has existed and was studied for a long time and early discussions on the topic can be found in the works of Ore \cite{ore1962theory} and Berge \cite{berge1973graphs}. Garey and Johnson \cite{gray1979computers} have shown that determining the domination number of an arbitrary graph is an NP-complete problem. 

The domination number can be defined equivalently by means of a function, which can be considered as a characteristic function of a dominating set (see \cite{haynes2017domination}). A function $f: V(G) \rightarrow \set{0,1}$ is called a \textit{domination function} on $G$ if for each vertex $u\in V(G), \sum_{v\in N_G[u]}f(v)\geq 1$. The value $w(f)=\sum_{u\in V(G)}f(u)$ is called the \textit{weight} of $f$. Now, the domination number of $G$ can be defined as
$$ \gamma(G)= \min\set{w(f): f~\text{is a domination function of G}}.$$

Analogously, we define a \textit{signed domination function} of $G$ to be a labelling of the vertices of $G$ with $+1$ and $-1$ such that the closed neighbourhood of each vertex contains more $+1$'s than $-1$'s. The signed domination number of $G$ is the minimum value of the sum of vertex labels that is taken over all \textit{signed domination function} of $G$. In this paper, we concentrate on a related function, called the signed Roman domination function, which we define now.
\begin{definition}
  Let $G=(V, E)$ be a graph. A \textit{signed Roman domination function} (SRDF) on the graph $G$ is a function $f: V(G)\rightarrow \set{-1, 1, 2}$ which satisfies the following conditions:
  \begin{enumerate}[(i)]
   \item For each $u\in V(G)$, $\sum_{v\in N_G[u]}f(v)\geq 1$, and
   \item Each vertex $u$ for which $f(u) = -1$ is adjacent to at least one vertex $v$ for which $f(v) = 2$.
\end{enumerate}
The value $f(V) = \sum_{u\in V(G)}f(u)$ is called the \textit{weight} of the function $f$ and is denoted by $w(f)$. The \textit{signed Roman domination number} of $G$,  $\gamma_{SR}(G)$, is the minimum weight of a SRDF on $G$.
\end{definition}

The concept of a SRDF was introduced by Ahangar et al. \cite{abdollahzadeh2014signed}. They described the usefulness of these concepts in various applicable areas (see \cite{abdollahzadeh2014signed}, \cite{henning2003defending} and \cite{stewart1999defend} for more details). It is obvious that for every graph $G$ of order $n$ we have $\gamma_{SR}(G) \leq n$, because assigning +1 to each vertex yields a SRDF. In \cite{abdollahzadeh2014signed} Ahangar et al. presented various lower and upper bounds on the signed Roman domination number of a graph in terms of its order, size and vertex degrees. For instance, they showed that for a graph $G$ with $n$ vertices, we have
\begin{equation}\label{eq:aa}
 \gamma_{SR}(G)\geq \left(\frac{-2\Delta^2+2\Delta\delta+\Delta+2\delta+3}{(\Delta+1)(2\Delta+\delta+3)}\right)n.
\end{equation}Moreover, they also showed that if $G$ is a graph with $n$ vertices and $m$ edges with no isolated vertex, then we have
\begin{equation}\label{eq:aaa}
    \gamma_{SR}(G)\geq \frac{3n-4m}{2}.
\end{equation}They investigated the relation between $\gamma_{SR}$ and some other graphical parameters, and the signed Roman domination number of some special bipartite graphs. It is proved in \cite{abdollahzadeh2014signed} that  $\gamma_{SR}(K_n)=1$ for each $n\neq 3$, $\gamma_{SR}(K_3)=2, \gamma_{SR}(C_n)=\left\lceil\dfrac{2n}{3}\right\rceil, \gamma_{SR}(P_n)= \left\lfloor\dfrac{2n}{3}\right\rfloor$, and that the only $n$-vertex graph $G$ with $ \gamma_{SR}(G)=n$ is the empty graph $\overline{K_n}$.

Note that each SRDF $f$ of $G$ is uniquely determined by the ordered partition $(V_{-1}, V_1, V_2)$ of $V(G)$, where $V_i = \set{u \in V(G) : f(u) = i}$ for each $i \in\set{ {-1, +1, 2}}$.
Specially, $w(f) = 2|V_2| + |V_1|- |V_{-1}|$. For simplicity, we usually write $f = (V_{-1}, V_1, V_2)$ and,
when $U \subseteq V (G)$ we denote the sum $\sum_{u\in U}f(u)$ by $ f(U)$.  If $w(f) = \gamma_{SR} (G)$, then $f$ is
called a $\gamma_{SR} (G)$-\textit{function} (also known as an \textit{optimal SRDF}) on $G$. There have been several follow up work on SRDFs after the work of Ahangar et al. \cite{abdollahzadeh2014signed}. For instance, Behtoei, Vatandoost and Azizi \cite{BehtoeiVatandoostAzizi} studied the signed Roman domination number of the join of graphs, and determined its value for the join of cycles, wheels, fans, and friendship graphs; while, Hong et. al. \cite{HongYuZhaZhang} determined its value spider graphs and double star graphs.

In this paper, we investigate the signed Roman domination number of the Ladder graph $LG_n$, Circular Ladder graph $LC_n$ and their complement graphs $LG^C_n$ and $LC_n^C$. The paper is arranged as follows: in Section \ref{sec:lg} we evaluate the signed Roman domination number for $LG_n$ and its complement, in Section \ref{sec:cl} we evaluate the signed Roman domination number for $LC_n$ and its complement, we end the paper with some concluding remarks in Section \ref{sec:conc}.

\section {Signed Roman domination number for Ladder graphs \texorpdfstring{$LG_n$}{LGn} and its complement}\label{sec:lg}

 A \textit{Cartesian product} of two graphs  $G_1$ and $G_2$ is the graph  $G_1 \Cross G_2$  with the vertex set $V(G_1) \Cross V(G_2)$ in which two vertices $(u_1, v_1)$ and $(u_2, v_2)$ are adjacent in $G_1 \Cross G_2$ if and only if either $u_1 = u_2$ and ${v_1 v_2} \in E(G_2)$ or $v_1 = v_2 $ and ${u_1 u_2} \in E(G_1)$. We are interested in the following cartesian product in this section.

\begin{definition}\label{laddergraph} The {\emph{Ladder graph}} of order $2n$, denoted by $LG_{n}$, is the Cartesian product of two path graphs, one of which has only one edge. In other words $LG_{n}:=~P_{2}\Cross P_{n}$, where $P_n$ is the path graph on $n$ vertices.
\end{definition}

Note, by the definition of $LG_n$, one can write
\begin{equation*}\label{ladvertset}
\centering
\begin{aligned}
V(LG_{n})&~=~\left\{(1,i), (2,i)|~i\in[n]\right\},~\text{and}\\
E(LG_{n})&~=~\{\{(1,i),(2,i)\}|~1\leq i \leq n\}\cup \{\{(1,i),(1,i+1)\},\{(2,i),(2,i+1)\}|~1\leq i\leq n-1\}.
\end{aligned}
\end{equation*}
For simplicity, we encode the vertices $(1,i)$ and $(2,i)$ with $1i$ and $2i$ for $i=1,2,\ldots,n$.

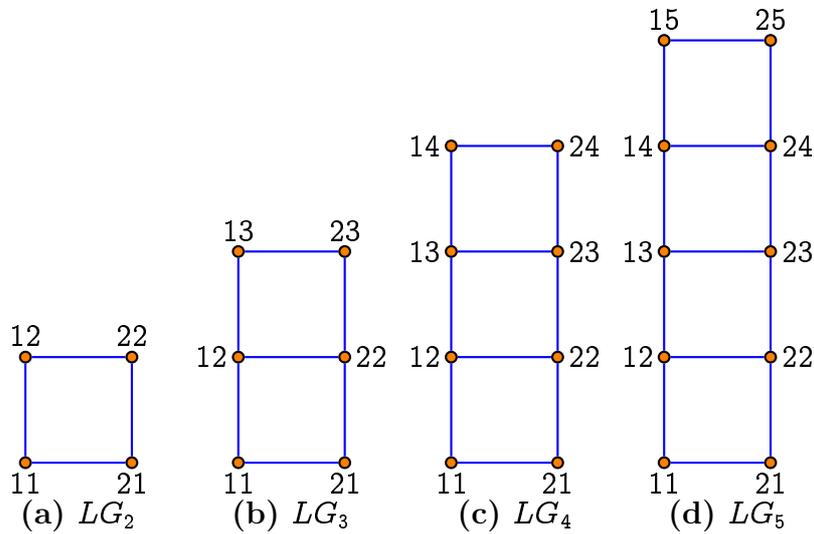
\begin{figure}[htb!]
\begin{center}
\begin{tikzpicture}
[inner sep=0.5mm, place/.style={circle,draw=black,fill=orange,thick},scale=0.70]
\node[place] (v_1) at (-2,-1) [label=below:$11$] {};
\node[place] (v_2) at (0,-1) [label=below:$21$] {}edge [-,thick,blue](v_1);
\node[place] (v_3) at (-2,1) [label=above:$12$] {}edge [-,thick,blue](v_1);
\node[place] (v_4) at (0,1) [label=above:$22$] {}edge [-,thick,blue](v_3) edge [-,thick,blue](v_2);
\node (dots) at (-1,-2) [label=center:{\bf (a) $LG_2$}]{};
[inner sep=0.5mm, place/.style={circle,draw=black,fill=orange,thick}]
\node[place] (v_1) at (2,-1) [label=below:$11$] {};
\node[place] (v_2) at (4,-1) [label=below:$21$] {}edge [-,thick,blue](v_1);
\node[place] (v_3) at (2,1) [label=left:$12$] {}edge [-,thick,blue](v_1);
\node[place] (v_4) at (4,1) [label=right:$22$] {}edge [-,thick,blue](v_3) edge [-,thick,blue](v_2);
\node[place] (v_5) at (2,3) [label=above:$13$] {}edge [-,thick,blue](v_3);
\node[place] (v_6) at (4,3) [label=above:$23$] {}edge [-,thick,blue](v_4) edge [-,thick,blue](v_5);
\node (dots) at (3,-2) [label=center:{\bf (b) $LG_3$}]{};
[inner sep=0.5mm, place/.style={circle,draw=black,fill=orange,thick}]
\node[place] (v_1) at (6,-1) [label=below:$11$] {};
\node[place] (v_2) at (8,-1) [label=below:$21$] {}edge [-,thick,blue](v_1);
\node[place] (v_3) at (6,1) [label=left:$12$] {}edge [-,thick,blue](v_1);
\node[place] (v_4) at (8,1) [label=right:$22$] {}edge [-,thick,blue](v_3) edge [-,thick,blue](v_2);
\node[place] (v_5) at (6,3) [label=left:$13$] {}edge [-,thick,blue](v_3);
\node[place] (v_6) at (8,3) [label=right:$23$] {}edge [-,thick,blue](v_4) edge [-,thick,blue](v_5);
\node[place] (v_7) at (6,5) [label=left:$14$] {}edge [-,thick,blue](v_5);
\node[place] (v_8) at (8,5) [label=right:$24$] {}edge [-,thick,blue](v_6) edge [-,thick,blue](v_7);
\node (dots) at (7.2,-2) [label=center:{\bf (c) $LG_4$}]{};
[inner sep=0.5mm, place/.style={circle,draw=black,fill=orange,thick}]
\node[place] (v_1) at (10,-1) [label=below:$11$] {};
\node[place] (v_2) at (12,-1) [label=below:$21$] {}edge [-,thick,blue](v_1);
\node[place] (v_3) at (10,1) [label=left:$12$] {}edge [-,thick,blue](v_1);
\node[place] (v_4) at (12,1) [label=right:$22$] {}edge [-,thick,blue](v_3) edge [-,thick,blue](v_2);
\node[place] (v_5) at (10,3) [label=left:$13$] {}edge [-,thick,blue](v_3);
\node[place] (v_6) at (12,3) [label=right:$23$] {}edge [-,thick,blue](v_4) edge [-,thick,blue](v_5);
\node[place] (v_7) at (10,5) [label=left:$14$] {}edge [-,thick,blue](v_5);
\node[place] (v_8) at (12,5) [label=right:$24$] {}edge [-,thick,blue](v_6) edge [-,thick,blue](v_7);
\node[place] (v_9) at (10,7) [label=above:$15$] {}edge [-,thick,blue](v_7);
\node[place] (v_10) at (12,7) [label=above:$25$] {}edge [-,thick,blue](v_8) edge [-,thick,blue](v_9);
\node (dots) at (11.2,-2) [label=center:{\bf (d) $LG_5$}]{};
\end{tikzpicture}
\end{center}
    \caption{Ladder graphs of orders $2, 3, 4$ and $5$.}
   \label{fig:p2p6p7}
\end{figure}

It is easy to observe that, for $LG_n$ we have
\begin{itemize}
\item[(i)] $|E(LG_{n})|=3n-2$, and
\item[(ii)] $LG_{n}$ has four vertices of degree $2$ and $2n-4$ vertices of degree $3$.
\end{itemize}

We now proceed to find the signed Roman domination number for Ladder graphs.
\begin{figure}
\begin{center}
\begin{tikzpicture}
[inner sep=0.5mm, place/.style={circle,draw=black,fill=orange,thick},scale=0.70]
\node[place] (v_1) at (-2,-1) [label=below:$-1$] {};
\node[place] (v_2) at (0,-1) [label=below:$1$] {}edge [-,thick,blue](v_1);
\node[place] (v_3) at (-2,1) [label=left:$2$] {}edge [-,thick,blue](v_1);
\node[place] (v_4) at (0,1) [label=right:$1$] {}edge [-,thick,blue](v_3) edge [-,thick,blue](v_2);
\node (dots) at (-1,-2) [label=center:{\bf (a) $LG_2$}]{};
[inner sep=0.5mm, place/.style={circle,draw=black,fill=orange,thick}]
\node[place] (v_1) at (2,-1) [label=below:$-1$] {};
\node[place] (v_2) at (4,-1) [label=below:$1$] {}edge [-,thick,blue](v_1);
\node[place] (v_3) at (2,1) [label=left:$2$] {}edge [-,thick,blue](v_1);
\node[place] (v_4) at (4,1) [label=right:$1$] {}edge [-,thick,blue](v_3) edge [-,thick,blue](v_2);
\node[place] (v_5) at (2,3) [label=left:$-1$] {}edge [-,thick,blue](v_3);
\node[place] (v_6) at (4,3) [label=right:$1$] {}edge [-,thick,blue](v_4) edge [-,thick,blue](v_5);
\node (dots) at (3,-2) [label=center:{\bf (b) $LG_3$}]{};

[inner sep=0.5mm, place/.style={circle,draw=black,fill=orange,thick}]
\node[place] (v_1) at (6,-1) [label=below:$-1$] {};
\node[place] (v_2) at (8,-1) [label=below:$1$] {}edge [-,thick,blue](v_1);
\node[place] (v_3) at (6,1) [label=left:$2$] {}edge [-,thick,blue](v_1);
\node[place] (v_4) at (8,1) [label=right:$1$] {}edge [-,thick,blue](v_3) edge [-,thick,blue](v_2);
\node[place] (v_5) at (6,3) [label=left:$-1$] {}edge [-,thick,blue](v_3);
\node[place] (v_6) at (8,3) [label=right:$-1$] {}edge [-,thick,blue](v_4) edge [-,thick,blue](v_5);
\node[place] (v_7) at (6,5) [label=left:$1$] {}edge [-,thick,blue](v_5);
\node[place] (v_8) at (8,5) [label=right:$2$] {}edge [-,thick,blue](v_6) edge [-,thick,blue](v_7);
\node (dots) at (7.2,-2) [label=center:{\bf (c) $LG_4$}]{};
[inner sep=0.5mm, place/.style={circle,draw=black,fill=orange,thick}]
\node[place] (v_1) at (10,-1) [label=below:$-1$] {};
\node[place] (v_2) at (12,-1) [label=below:$1$] {}edge [-,thick,blue](v_1);
\node[place] (v_3) at (10,1) [label=left:$2$] {}edge [-,thick,blue](v_1);
\node[place] (v_4) at (12,1) [label=right:$1$] {}edge [-,thick,blue](v_3) edge [-,thick,blue](v_2);
\node[place] (v_5) at (10,3) [label=left:$-1$] {}edge [-,thick,blue](v_3);
\node[place] (v_6) at (12,3) [label=right:$-1$] {}edge [-,thick,blue](v_4) edge [-,thick,blue](v_5);
\node[place] (v_7) at (10,5) [label=left:$1$] {}edge [-,thick,blue](v_5);
\node[place] (v_8) at (12,5) [label=right:$2$] {}edge [-,thick,blue](v_6) edge [-,thick,blue](v_7);
\node[place] (v_9) at (10,7) [label=left:$1$] {}edge [-,thick,blue](v_7);
\node[place] (v_10) at (12,7) [label=right:$-1$] {}edge [-,thick,blue](v_8) edge [-,thick,blue](v_9);
\node (dots) at (11.2,-2) [label=center:{\bf (d) $LG_5$}]{};
\end{tikzpicture}
\end{center}
\caption{The SRDF for the Ladder graphs of order $2, 3, 4$ and $5$.}\label{fig:srdflader25}
\end{figure}
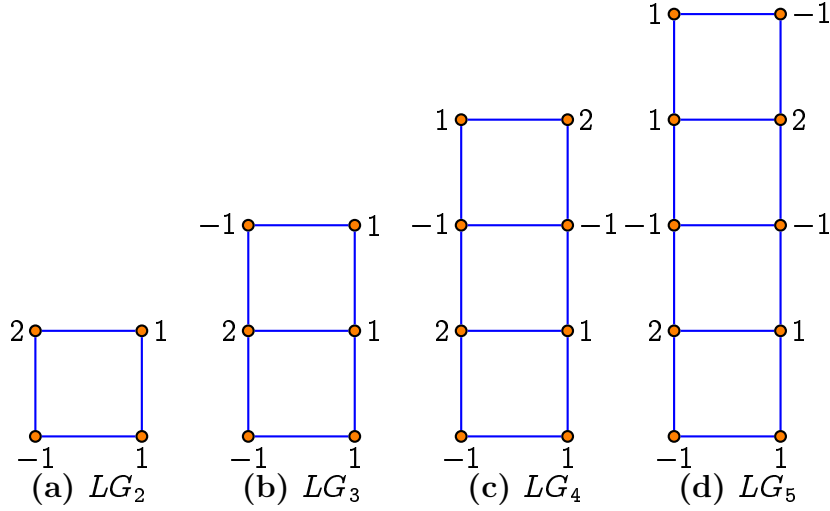

\begin{theorem}\label{thm:1}
Let $LG_{n}$ be the Ladder graph of order $2n$ with $n\geq2$. Then 
$$ \gamma_{SR} (LG_{n})=  \left\lfloor\dfrac{n+2}{2}\right\rfloor +1.$$
\end{theorem}

\noindent From \eqref{eq:aa} we get a bound of
\[
\gamma_{SR}(LG_n)\geq \frac{2n}{11}.
\]

\begin{proof}[Proof of Theorem \ref{thm:1}]
For the graph $LG_{n}$ consider the function \[f:V(LG_{n}) \rightarrow \set{-1,1,2}\] given by
$$f(v)=\begin{cases}-1      &   \text{if } v=(1,i), i\geq 1 \text { and } i \text{ is odd, except when $i=n\equiv 1 \pmod 4$;}\\ & v=(2,i), i \geq 3 \text{ and } i \text{ is odd, except when $i=n\equiv 3 \pmod 4$;}\\
                                    1      &    \text{if } v=(2,1), v=(1,i) \text{ and } i  \equiv 0 \pmod {4}, v=(1,n)\text{ for $n\equiv 1 \pmod 4$;} \\ & v=(2,i) \text{ and } i  \equiv 2\pmod {4}, v=(2,n) \text{ and } n  \equiv  2, 3 \pmod {4}; \\
                                    2      &  \text{if }  v=(1,i) \text{ and } i  \equiv 2 \pmod {4},  v=(2,i) \text{ and } i  \equiv 0 \pmod {4}.
           \end{cases} $$
\newline
By our construction, it is easy to see that $f$ is a SRDF. Except for the $(1,1), (2,1), (1,n)$ and $(2,n)$ vertices, every vertex with label $2$ is adjacent to three vertices, two of which are labelled $-1$ and one of them is labelled $1$; every vertex with label $-1$ is adjacent to three vertices, labelled $1, 2$ and $-1$; and, every vertex with label $1$ is adjacent to three vertices, two of which are labelled $-1$ and one of them is labelled $2$. This establishes the conditions for the function be a SRDF\footnote{We write this explanation explicitly here, but from now on we will omit such an explanation as the next cases are similar.}. By the construction of our labels, it is easy to verify the following properties:
\begin{equation*}
	\begin{aligned}
|V_{-1}|&= n-1,\\
|V_1|&= \begin{cases}      \lfloor n/2\rfloor +1, & \text{if }n  \text { is even,}  \\
                                            \lfloor n/2\rfloor +2, &    \text{if }n \text { is odd,} 
           \end{cases} \\
|V_2|&=  \left\lfloor\dfrac{n}{2}\right\rfloor.
\end{aligned}
\end{equation*}

From now on, we will use the following fact without commentary: for all $n\geq 0$, we have 
$$\left\lfloor\dfrac{n}{2}\right\rfloor = \begin{cases}      n/2, & \text{if }n  \text { is even,}  \\
                                            (n-1)/2, &    \text{if }n \text { is odd.}
           \end{cases} $$
With this in hand, it is easy to see that the weight of the function $f$ defined above is given by
\[
2|V_2|+|V_1|-|V_{-1}|=\left\lfloor\dfrac{n+2}{2}\right\rfloor+1.
\]
Hence, we conclude that $ \gamma_{SR} (LG_{n})\leq  \left\lfloor\dfrac{n+2}{2}\right\rfloor +1  $, for all $n\geq 2$.

We now want to show that this is the minimum weight. Assume that \[g: V (LG_{n}) \rightarrow \set{-1,1,2}\] is another arbitrary SRDF  on $LG_{n}$. By the construction of the Ladder graph there are $4$ vertices of degree $2$, namely, the vertices $x=(1,1), y=(2,1), z=(1,n), \text{ and } w=(2,n)$. Clearly it is not possible to have $g(x)= g(y) = -1$, as this would violate the first condition in the definition of a SRDF. A similar argument also holds for the vertices $z$ and $w$. So we must have one of the following cases:
\begin{itemize}
    \item $g(x)=-1$ and $g(y)=1$ or $g(x)=1$ and $g(y)=-1$,
    \item $g(x)= g(y)=1$,
    \item $g(x)=1$ and $g(y)=2$ or $g(x)=2$ and $g(y)=1$,
    \item $g(x)=-1$ and $g(y)=2$ or $g(x)=2$ and $g(y)=-1$, and
    \item $g(x)= g(y)=2$.
\end{itemize}

It is easy to check that following the greedy algorithm, each of the above cases leads to the following inequalities:
\begin{equation*}
	\begin{aligned}
 |V_{-1}|&\leq (n-1),\\
|V_1|&\geq \begin{cases}      \lfloor n/2\rfloor +1, & \text{if }n  \text { is even},  \\
                                               \left\lfloor n/2\right\rfloor +2, &    \text{if }n \text { is odd,} 
           \end{cases}\\
|V_2|&\geq  \left\lfloor\dfrac{n}{2}\right\rfloor .
\end{aligned}
\end{equation*}
Here by greedy algorithm, we mean a constructive step by step approach where we assign our new label to a vertex by considering two key points: Firstly, the new label should comply with the SRDF conditions, and secondly, we must select the minimum possible value from the set $\set{-1,1,2}$ with respect to the previous labeled vertices. This approach will be clearer by doing an example:

Let's consider the ladder graph $LG_6$. Following our recipe we have the following SRDF labelling:

$$f(v)=\begin{cases}-1      &   \text{if } v=(1,i), i\geq 1 \text { and } i \text{ is odd, except when $i=n\equiv 1 \pmod 4$;}\\ & v=(2,i), i \geq 3 \text{ and } i \text{ is odd, except when $i=n\equiv 3 \pmod 4$;}\\
                                    1      &    \text{if } v=(2,1), v=(1,i) \text{ and } i  \equiv 0 \pmod {4}, v=(1,n)\text{ for $n\equiv 1 \pmod 4$;} \\ & v=(2,i) \text{ and } i  \equiv 2\pmod {4}, v=(2,n) \text{ and } n  \equiv  2, 3 \pmod {4}; \\
                                    2      &  \text{if }  v=(1,i) \text{ and } i  \equiv 2 \pmod {4},  v=(2,i) \text{ and } i  \equiv 0 \pmod {4}.
           \end{cases} $$

$$ \gamma_{SR} (LG_{6})=  \left\lfloor\dfrac{n+2}{2}\right\rfloor +1 = \left\lfloor\dfrac{6+2}{2}\right\rfloor +1 =5.$$

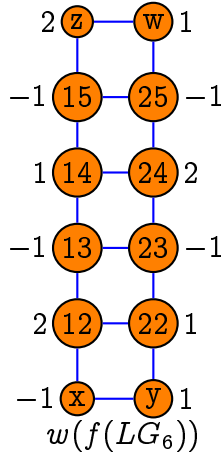
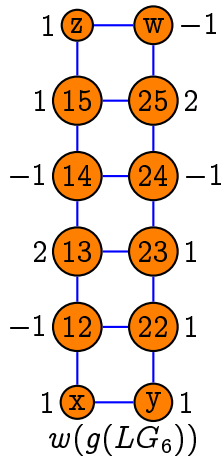
\begin{figure}
\begin{center}
\begin{tikzpicture}
[inner sep=0.5mm, place/.style={circle,draw=black,fill=orange,thick},scale=0.50]
[inner sep=0.5mm, place/.style={circle,draw=black,fill=orange,thick}]
\node[place] (v_1) at (10,-1) [label=left:$-1$] {x};
\node[place] (v_2) at (12,-1) [label=right:$1$] {y}edge [-,thick,blue](v_1);
\node[place] (v_3) at (10,1) [label=left:$2$] {12}edge [-,thick,blue](v_1);
\node[place] (v_4) at (12,1) [label=right:$1$] {22}edge [-,thick,blue](v_3) edge [-,thick,blue](v_2);
\node[place] (v_5) at (10,3) [label=left:$-1$] {13}edge [-,thick,blue](v_3);
\node[place] (v_6) at (12,3) [label=right:$-1$] {23}edge [-,thick,blue](v_4) edge [-,thick,blue](v_5);
\node[place] (v_7) at (10,5) [label=left:$1$] {14}edge [-,thick,blue](v_5);
\node[place] (v_8) at (12,5) [label=right:$2$] {24}edge [-,thick,blue](v_6) edge [-,thick,blue](v_7);
\node[place] (v_9) at (10,7) [label=left:$-1$] {15}edge [-,thick,blue](v_7);
\node[place] (v_10) at (12,7) [label=right:$-1$] {25}edge [-,thick,blue](v_8) edge [-,thick,blue](v_9);
\node[place] (v_11) at (10,9) [label=left:$2$] {z}edge  [-,thick,blue](v_9);
\node[place] (v_12) at (12,9) [label=right:$1$] {w}edge [-,thick,blue](v_11) edge [-,thick,blue](v_10);
\node (dots) at (11.2,-2) [label=center:{\bf  $w(f(LG_6))$}]{};
\end{tikzpicture}
\end{center}
\end{figure}
\newpage
Now, let's assume another $g:V(LG_6)\rightarrow \set{-1,1,2}$ where $g(x)=g(y)=1$ we follow the greedy algorithm for this case as an illustration.

Just keep in mind that we construct our labeling under greedy algorithm with respect to $g$ in such a way that in each step we keep the two SRDF conditions satisfied and that we choose the minimum value for labeling our next vertices. Now we can assign $g(12)=-1$ or $g(22)=-1$ but not both since it violates the first condition of an SRDF. So we can assign $g(12)=-1$ and $g(22)=1$ to keep it minimum. 

Now the only option for vertex 13 is $g(13)=2$ since we already assigned $g(12)=-1$ and hence the second condition of SRDF must be satisfied. To keep our new labeling minimum, we have to assign $g(23)=1$. By following the same reasoning we can assign $g(14)=g(24)=-1$. For the next step, we can not assign $g(15)=-1$ since it violate the first condition of SRDF for vertex 14 with label $g(14)=-1$, so, to keep our g-labeling minimum we can put $g(15)=1$ and for vertex 25, we must assign $g(25)=2$ since $g(24)=-1$. Finally, for the vertices $w$ and $z$ we can assign $g(w)=-1$ and $g(z)=2$ to keep our labeling both valid and minimum with respect to the previous labelling.  

\begin{figure}
\begin{center}
\begin{tikzpicture}
[inner sep=0.5mm, place/.style={circle,draw=black,fill=orange,thick},scale=0.50]
[inner sep=0.5mm, place/.style={circle,draw=black,fill=orange,thick}]
\node[place] (v_1) at (10,-1) [label=left:$1$] {x};
\node[place] (v_2) at (12,-1) [label=right:$1$] {y}edge [-,thick,blue](v_1);
\node[place] (v_3) at (10,1) [label=left:$ -1$] {12}edge [-,thick,blue](v_1);
\node[place] (v_4) at (12,1) [label=right:$ 1$] {22}edge [-,thick,blue](v_3) edge [-,thick,blue](v_2);
\node[place] (v_5) at (10,3) [label=left:$ 2$] {13}edge [-,thick,blue](v_3);
\node[place] (v_6) at (12,3) [label=right:$ 1$] {23}edge [-,thick,blue](v_4) edge [-,thick,blue](v_5);
\node[place] (v_7) at (10,5) [label=left:$ -1$] {14}edge [-,thick,blue](v_5);
\node[place] (v_8) at (12,5) [label=right:$-1 $] {24}edge [-,thick,blue](v_6) edge [-,thick,blue](v_7);
\node[place] (v_9) at (10,7) [label=left:$ 1$] {15}edge [-,thick,blue](v_7);
\node[place] (v_10) at (12,7) [label=right:$2 $] {25}edge [-,thick,blue](v_8) edge [-,thick,blue](v_9);
\node[place] (v_11) at (10,9) [label=left:$1 $] {z}edge  [-,thick,blue](v_9);
\node[place] (v_12) at (12,9) [label=right:$ -1$] {w}edge [-,thick,blue](v_11) edge [-,thick,blue](v_10);
\node (dots) at (11.2,-2) [label=center:{\bf  $w(g(LG_6))$}]{};
\end{tikzpicture}
\end{center}
\end{figure}
Now, we can see that $$w(g)=6> 5=w(f).$$ This illustrate a greedy algorithm approach for the case where $g(x)=g(y)=1.$ The arguments for the other cases are very similar. 

Consequently, we have 
\begin{equation*}
w(g) = 2|V_2|+|V_1|-|V_{-1}|\geq \left\lfloor\dfrac{n+2}{2}\right\rfloor+1
\end{equation*}
This completes the proof. 
 \end{proof}


Now, we turn our attention to the complement of the Ladder graph $LG^C_n$ and determine their signed Roman domination number. As an illustration, the complement of $LG^C_n$ when $n=2$ and $3$ is shown in Figure \ref{fig:cp2p2p3} with $\gamma_{SR}{ (LG^C_2)}=2$ and $\gamma_{SR}{ (LG^C_3)}=3$.  

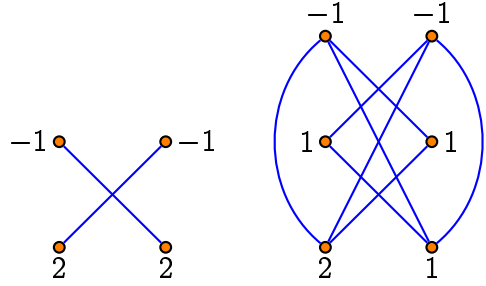
\begin{figure}[htb]
\begin{tikzpicture}
[inner sep=0.5mm, place/.style={circle,draw=black,fill=orange,thick},scale=0.7]
\node[place] (v_11) at (-7,-1) [label=below:$2$] {};
\node[place] (v_22) at (-5,-1) [label=below:$2$] {};
\node[place] (v_33) at (-7,1) [label=left:$-1$] {};
\node[place] (v_44) at (-5,1) [label=right:$-1$] {};
\tikzset{EdgeStyle/.append style = {blue, bend left = 50}}
\tikzset{EdgeStyle/.append style = {blue, bend right = 50}}
\tikzset{EdgeStyle/.append style = {blue, bend left=0}}
\Edge[](v_11)(v_44), 
\Edge[](v_22)(v_33), 
 [inner sep=0.5mm, place/.style={circle,draw=black,fill=orange,thick}]
\node[place] (v_1) at (-2,-1) [label=below:$2$] {};
\node[place] (v_2) at (0,-1) [label=below:$1$] {};
\node[place] (v_3) at (-2,1) [label=left:$1$] {};
\node[place] (v_4) at (0,1) [label=right:$1$] {};
\node[place] (v_5) at (-2,3) [label=above:$-1$] {};
\node[place] (v_6) at (0,3) [label=above:$-1$] {};
\tikzset{EdgeStyle/.append style = {blue, bend left = 50}}
\Edge[](v_1)(v_5)
\tikzset{EdgeStyle/.append style = {blue, bend right = 50}}
\Edge[](v_2)(v_6)
\tikzset{EdgeStyle/.append style = {blue, bend left=0}}
\Edge[](v_1)(v_4), \Edge[](v_1)(v_6), \Edge[](v_2)(v_3), \Edge[](v_2)(v_5), \Edge[](v_3)(v_6), \Edge[](v_4)(v_5)
\end{tikzpicture}
 \caption{SRDF for ${LG^C_2}$ and ${LG^C_3}$.}
 \label{fig:cp2p2p3}
\end{figure}

\begin{theorem}\label{thm2}
Let $ LG^C_n$ be complement of the Ladder graph. Then, for all $n\geq 4$, we have $$ \gamma_{SR}{ (LG^C_n)}=2.$$
\end{theorem}

From \eqref{eq:aa}, we get the bound of $\gamma_{SR}(LG^C_n)>1$ for $n\geq 2$.

\begin{proof}[Proof of Theorem \ref{thm2}]
For the graph ${ LG^C_n}$, consider the function \[f: V( LG^C_n )\rightarrow \set{-1, 1, 2}\] given by 
$$f(v)=\begin{cases}               -1      &   \text{if } v=(1,i),  2\leq i\leq n-1, v=(2,2) \text{ and }v=(2,n-1);\\
                                    1      &    \text{if } v=(2,i) \text{ with } i\neq 2, n-1;\\ 
                                    2      &   \text{if } v=(1,1) \text{ and } v=(1,n).
           \end{cases} $$
Clearly, $f$ is a signed Roman domination function. By the construction of our labels, it is easy to verify the following
\[
|V_{-1}|=n, \quad |V_1|=n-2, \quad \text{and} \quad |V_2|=2.
\]
Then, we have
\[
\gamma_{SR} {(LG^C_n)}\leq 2|V_2|+|V_1|-|V_{-1}|=2.
\]

We now show that this is the minimum weight function. Let us assume that $$g: V(LG^C_n) \rightarrow \set{-1, 1, 2}$$ is another arbitrary SRDF on $LG^C_n,~\text{where}~ n\geq 4$. By the construction of $LG^C_n$, there are four vertices of degree $2n-3$ and the other $2n-4$ internal vertices are of degree $2n-4$. Let the four vertices that have degree $2n-3$ be $x=(1,1), y=(2,1), z=(1,n), \text{ and } w=(2,n)$. 

Clearly $g(x)=g(y)=1$ is not possible as this would violate the second condition of being a SRDF, as we know $g(1,i)=-1$ for $2\leq i\leq n-1$ in this case, otherwise we will get $w(g) > w(f)$. Similarly we can conclude that $g(w) = g(z)=1$ is not possible. 

If $g(y)=-1$, since $y$ is adjacent to all vertices $v=(1,i)$ where $2\leq i\leq n$ and $v=(2,i)$ where $3\leq i\leq n,$ in $LG_n^C$, we have
\begin{equation*}
\begin{split}
 \sum_{v\in N_{LG^C_n}[y]}g(v) &= \sum_{i=2}^n g(1,i) + \sum_{i=1, ~i\neq 2}^n g(2,i)\\
 & = (-1)\cdot(n-2)+2+1\cdot(n-3)-2 <0.
\end{split}
\end{equation*}
This clearly violates the second condition for $g$ being a SRDF. Similarly if $g(2,i)=-1$ for $3\leq i \leq n$ and $i\neq n-1$, we will get $\sum_{v\in N_{LG^C_n}[y]}g(v)<1$, which also violates the second condition for $g$ being a SRDF. This completes the proof.
\end{proof}

\section {Signed Roman domination number for Ladder graphs Circular Ladder graph and its complement}\label{sec:cl}

In this section we look at a different cartesian product of graphs, which we define below.
\begin{definition}\label{circladdergraph} The {\emph{circular Ladder graph}} of order $2n$, denoted by $LC_{n}$, is the Cartesian product of a cycle graph $C_n,~ n\geq3$ and an edge $P_2$. In other words:
\begin{equation*}\label{circladeq1}
\centering
LC_{n}:=~C_{n}\Cross P_{2}.
\end{equation*}
\end{definition}
It is easy to check that $LC_{n}$ has $2n$ vertices and $3n$ edges. As an illustration, the circular Ladder graphs of orders $3, 4$ and $5$ are shown in Figure~\ref{fig:circlader25}. For simplicity, we encode the vertices $(1,i)$ and $(2,i)$ with $1i$ and $2i$ for $i=1,2,\ldots ,n$.

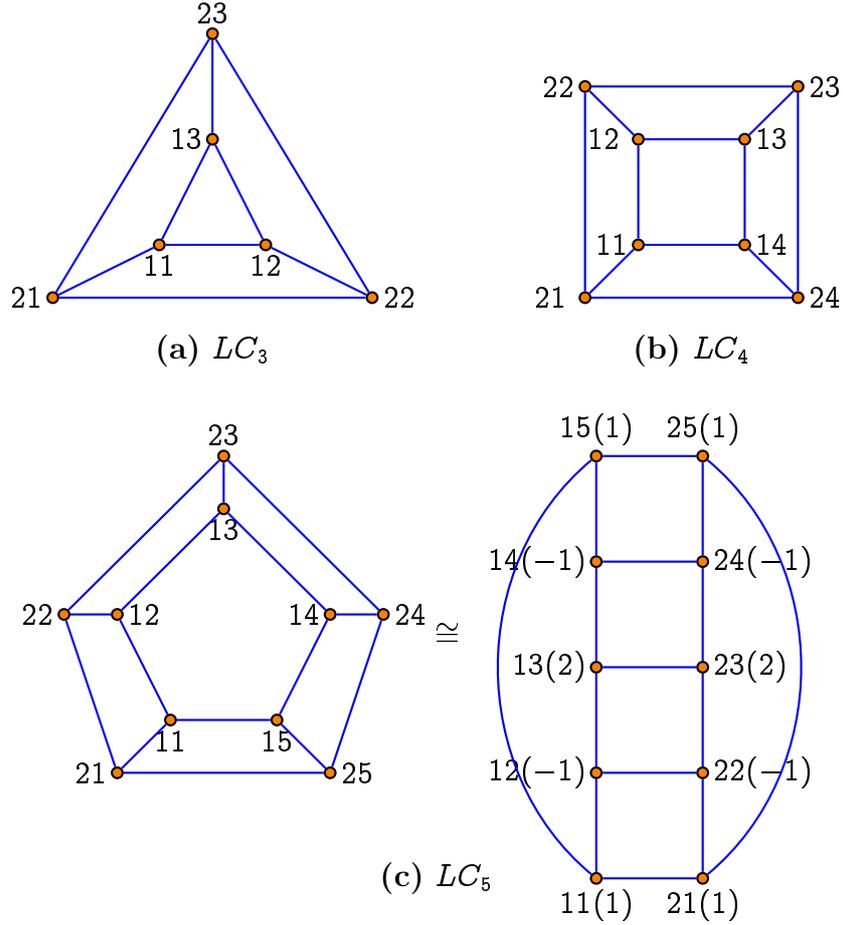
\begin{figure}[htbp]
\begin{center}
\begin{tikzpicture}
[inner sep=0.5mm, place/.style={circle,draw=black,fill=orange,thick},scale=0.7]
\node[place] (v_1) at (-4,-1) [label=below:$11$] {};
\node[place] (v_2) at (-3,1) [label=left:$13$] {}edge [-,thick,blue](v_1);
\node[place] (v_3) at (-2,-1) [label=below:$12$] {}edge [-,thick,blue](v_2) edge [-,thick,blue](v_1);
\node[place] (v_4) at (-6,-2) [label=left:$21$] {}edge [-,thick,blue](v_1) ;
\node[place] (v_5) at (-3,3) [label=above:$23$] {}edge [-,thick,blue](v_2) edge [-,thick,blue](v_4);
\node[place] (v_6) at (0,-2) [label=right:$22$] {}edge [-,thick,blue](v_3) edge [-,thick,blue](v_4) edge [-,thick,blue](v_5);
\node (dots) at (-3,-3) [label=center:{\bf (a) $LC_3$}]{};

[inner sep=0.5mm, place/.style={circle,draw=black,fill=orange,thick}]
\node[place] (v_1) at (5,-1) [label=left:$11$] {};
\node[place] (v_2) at (5,1) [label={[label distance=1mm]left:$12$}] {}edge [-,thick,blue](v_1);
\node[place] (v_3) at (7,1) [label=right:$13$] {}edge [-,thick,blue] (v_2);
\node[place] (v_4) at (7,-1) [label=right:$14$] {}edge [-,thick,blue] (v_3)edge [-,thick,blue] (v_1);
\node[place] (v_5) at (4,-2) [label={[label distance=1mm]left:$21$}] {}edge [-,thick,blue] (v_1);
\node[place] (v_6) at (4,2) [label=left:$22$] {}edge [-,thick,blue] (v_2) edge [-,thick,blue] (v_5);
\node[place] (v_7) at (8,2) [label=right:$23$] {}edge [-,thick,blue] (v_3) edge [-,thick,blue] (v_6);
\node[place] (v_8) at (8,-2) [label=right:$24$] {}edge [-,thick,blue] (v_4) edge [-,thick,blue] (v_5) edge [-,thick,blue] (v_7);
\node (dots) at (6,-3) [label=center:{\bf (b) $LC_4$}]{};
\node (dots) at (6,-4) [label=center:{}]{};
\end{tikzpicture}
\end{center}



\begin{center}
\begin{tikzpicture}
[inner sep=0.5mm, place/.style={circle,draw=black,fill=orange,thick},scale=0.7]
\node[place] (v_1) at (-10,-2) [label=below:$11$] {};
\node[place] (v_2) at (-11,0) [label=right:$12$] {}edge [-,thick,blue](v_1);
\node[place] (v_3) at (-9,2) [label=below:$13$] {}edge [-,thick,blue](v_2);
\node[place] (v_4) at (-7,0) [label=left:$14$] {}edge [-,thick,blue](v_3);
\node[place] (v_5) at (-8,-2) [label=below:$15$] {}edge [-,thick,blue](v_1) edge [-,thick,blue](v_4);
\node[place] (v_6) at (-11,-3) [label=left:$21$] {}edge [-,thick,blue](v_1);
\node[place] (v_7) at (-12,0) [label=left:$22$] {}edge [-,thick,blue](v_2) edge [-,thick,blue](v_6);
\node[place] (v_8) at (-9,3) [label=above:$23$] {}edge [-,thick,blue](v_3) edge [-,thick,blue](v_7);
\node[place] (v_9) at (-6,0) [label=right:$24$] {}edge [-,thick,blue](v_4) edge [-,thick,blue](v_8);
\node[place] (v_10) at (-7,-3) [label=right:$25$] {}edge [-,thick,blue](v_6) edge [-,thick,blue](v_5) edge [-,thick,blue](v_9);
\node (dots) at (-5,-5) [label=center:{\bf (c) $LC_5$}]{};
\node[]at (-4.8,0) [label=below:$\cong$] {};
[inner sep=0.5mm, place/.style={circle,draw=black,fill=orange,thick}]
\node[place] (v_1) at (0,-5) [label=below:$21(1)$] {};
\node[place] (v_2) at (-2,-5) [label=below:$11(1)$] {};
\node[place] (v_3) at (0,-3) [label=right:$22(-1)$] {};
\node[place] (v_4) at (-2,-3) [label=left:$12(-1)$] {};
\node[place] (v_5) at (0,-1) [label=right:$23(2)$] {};
\node[place] (v_6) at (-2,-1) [label=left:$13(2)$] {};
\node[place] (v_7) at (0,1) [label=right:$24(-1)$] {};
\node[place] (v_8) at (-2,1) [label=left:$14(-1)$] {};
\node[place] (v_9) at (0,3) [label=above:$25(1)$] {};
\node[place] (v_{10}) at (-2,3) [label=above:$15(1)$] {};
\tikzset{EdgeStyle/.append style = {blue, bend left=0}}
\Edge(v_1)(v_2), \Edge(v_1)(v_3), \Edge(v_2)(v_4), \Edge(v_3)(v_5), \Edge(v_3)(v_4), \Edge(v_4)(v_6), \Edge(v_5)(v_6),\Edge(v_5)(v_7), \Edge(v_7)(v_9), \Edge(v_6)(v_8), \Edge(v_8)(v_{10}),\Edge(v_7)(v_8), \Edge(v_9)(v_{10});
\tikzset{EdgeStyle/.append style = {blue, bend left = -50}}
\Edge(v_1)(v_9);
\tikzset{EdgeStyle/.append style = {blue, bend left = 50}}
\Edge(v_2)(v_{10});
\end{tikzpicture}
\caption{The circular Ladder graphs of orders $3, 4$ and $5$.}
\label{fig:circlader25}
\end{center}
\end{figure}
\begin{example}
For $LC_5$, the signed Roman domination number is equal to 4 (see Figure \ref{fig:circlader25}~(c)).
\end{example}

\begin{theorem}
Let $LC_n$ be the circular Ladder graph of the order $2n$. Then, for all $n\geq 3, ~n\neq 5$, we have
$$ \gamma_{SR} (LC_n)=\begin{cases}
            \left\lfloor\dfrac{n+2}{2}\right\rfloor+1,     & \text{if }n  \equiv 0,2,3 \pmod {4}, \\ &\\
            \left\lfloor\dfrac{n+2}{2}\right\rfloor+2,      & \text{if }n  \equiv 1 \pmod {4}.   
            \end{cases} $$
\end{theorem}

\noindent From \eqref{eq:aa} we get a bound of
\[
\gamma_{SR}(LG_n)\geq \frac{n}{4}.
\]

\begin{proof}
For the circular Ladder graph $LC_n, ~n\geq 3$ consider the function \[f:V(LC_n)\rightarrow \set{-1, 1, 2}\] given by
$$f(v)=\begin{cases}-1,      &   \text{if } v=(1,i) \text{ for } i\geq 1 \text { and } i \text{ is odd}; v=(2,i) \text { for } i \geq 3 \text{ and } i \text{ is odd ($i\neq n$).}\\
                                    1,      &    \text{if } v=(2,1); v=(1,i) \text{ and } i  \equiv 0\pmod {4}; v=(2,i) \text{ and } i  \equiv 2 \pmod {4}; \\ & v=(2,n) \text{ and } n  \equiv 3 \pmod {4}\\
                                    2,      & \text{if }  v=(1,i) \text{ and } i  \equiv 2 \pmod {4} ; v=(2,i) \text{ and } i  \equiv 0 \pmod {4}; \\ & v=(2,n) \text{ and } n  \equiv 0,1 \pmod {4}.
           \end{cases} $$
It is clear that $f$ is a SRDF. By our labeling, it is easy to show that
\begin{equation*}
\begin{aligned}
|V_{-1}|&= n-1,\\
|V_1|&=\begin{cases}      \left\lfloor\dfrac{n+2}{2}\right\rfloor, & \text{if }n \equiv 0,1,2 \pmod {4},  \\ &\\
                                                 \left\lceil\dfrac{n+2}{2}\right\rceil,  &  \text{if }n \equiv 3 \pmod {4}.
                     \end{cases} \\
|V_2|&=\begin{cases}      \left\lfloor\dfrac{n}{2}\right\rfloor,  & \text{if }n \equiv 0,2,3 \pmod {4},  \\ &\\
                                                 \left\lfloor\dfrac{n+2}{2}\right\rfloor,  &  \text{if }n \equiv 1 \pmod {4}.
                     \end{cases} \\
\end{aligned}
\end{equation*}

It is easy to verify the following inequality
\[
\gamma_{SR} {(LC_n)} \leq 2|V_2|+|V_1|-|V_{-1}|=\begin{cases}
            \left\lfloor\dfrac{n+2}{2}\right\rfloor+1,     & \text{if }n  \equiv 0,2,3 \pmod {4}, \\ &\\
            \left\lfloor\dfrac{n+2}{2}\right\rfloor+2,      & \text{if }n  \equiv 1 \pmod {4}.   
            \end{cases}
\]

We now show that our function is the minimum SRDF. Let us assume that\[g:V(LC_n)\rightarrow\set{-1,1,2}\] is another arbitrary 
SRDF on $LC_n$. By the construction of the circular Ladder graph all vertices have the same degree (that is, $LC_n$ is a $3$-regular graph). So, whether we select the vertices from the interior or exterior, it makes no difference. Therefore, we will choose the four exterior vertices, namely $x=(1,1), y=(2,1), w=(1,n)~\text{and}~z=(2,n)$.

Clearly, it is not possible to have the following cases
\begin{itemize}
    \item $g(x)=g(y)=g(w)=-1$,
    \item $~g(x)=g(y)=g(z)=-1$,
    \item $~g(x)=g(w)=g(z)=-1$,
    \item $~g(y)=g(z)=g(w)=-1$, and
    \item $g(x)=g(y)=g(w)=g(z)=-1$.
\end{itemize}
This is because if any of the above occurs, then we will have the following cases (in order)
\begin{itemize}
    \item $\sum_{v\in N_{LC_n}[x]}g(v)<0$,
    \item $\sum_{v\in N_{LC_n}[y]}g(v)<0$,
    \item $\sum_{v\in N_{LC_n}[w]}g(v)<0$,
    \item $\sum_{v\in N_{LC_n}[z]}g(v)<0$, and
    \item $\sum_{v\in N_{LC_n}[u]}g(v)<0$ for $u\in \{x,y,w,z\}$.
    \end{itemize}

So, we must have one of the following cases (some cases will be excluded by the symmetric role the vertices play):
\begin{itemize}
    \item $g(x)=g(y)=-1$ and $g(w)=g(z)=1$ (or, $g(x)=g(y)=1$ and $g(w)=g(z)=-1$),
    \item $g(x)=g(y)=g(z)=g(w)=1$,
    \item $g(x)=g(y)=-1$, $g(w)=1$ and $g(z)=2$ (or, $g(x)=g(y)=-1$, $g(w)=2$ and $g(z)=1$),
    \item $g(x)=g(y)=g(w)=g(z)=2$,
    \item $g(x)=g(y)=2$ and $g(w)=g(z)=1$,
    \item $g(x)=g(y)=g(w)=1$ and $g(z)=2$,
    \item $g(x)=g(y)=g(w)=2$ and $g(z)=1$,
    \item $g(x)=g(y)=2$ and $g(w)=g(z)=-1$; and
    \item $g(x)=g(y)=g(w)=2$ and $g(z)=-1$.
\end{itemize}

It is easy to check that following the greedy algorithm, each of the above cases give us
\begin{equation*}
\begin{aligned}
|V_{-1}|&\leq n-1,\\
|V_1|&\geq\begin{cases}      \left\lfloor\dfrac{n+2}{2}\right\rfloor, & \text{if }n \equiv 0,1,2 \pmod {4};  \\ &\\
                                                 \left\lceil\dfrac{n+2}{2}\right\rceil,  &  \text{if }n \equiv 3 \pmod {4},
                     \end{cases}\\
|V_2|&\geq\begin{cases}      \left\lfloor\dfrac{n}{2}\right\rfloor,  & \text{if }n \equiv 0,2,3 \pmod {4};  \\ &\\
                                                 \left\lfloor\dfrac{n+2}{2}\right\rfloor,  &  \text{if }n \equiv 1 \pmod {4}.
                     \end{cases}
\end{aligned}
\end{equation*}
Consequently, summing up these inequalities we have
\begin{equation*}
\begin{aligned}
w(g)=2|V_2|+|V_1|-|V_{-1}|&\geq\begin{cases}      \left\lfloor\dfrac{n+2}{2}\right\rfloor +1, & \text{if }n \equiv 0,2,3 \pmod {4}, \\ &\\
                                                 \left\lfloor\dfrac{n+2}{2}\right\rfloor +2, &  \text{if }n \equiv 1 \pmod {4}.
                     \end{cases} \\
\end{aligned}
\end{equation*}
This completes the proof. 
\end{proof}

Now, we turn our attention to the complement of the circular Ladder graph $LC_n^C$. 
\begin{definition}
    The complement of the circular Ladder graph (denoted by $LC_n^C$) is the complement of the cartesian product of path graph $P_2$ and cycle graph $C_n$, that is $$LC_n^C=(P_2\Cross C_n)^C.$$
\end{definition}
It is easy to see that for the complement of the circular Ladder graph of order $2n$ (that is $LC_n^C$), we have
  \begin{enumerate}[(i)]
   \item $|V(LC_n^C)| =2n$,
\item $|E(LC_n^C)|=\dfrac{2n(2n-1)}{2}-|E(LC_n)|$,
   \item$LC_n^C$ is a regular graph of  degree $(2n-4)$.
  \end{enumerate}

\begin{example}
The complement of the circular Ladder graph of order 3 and 4 has signed Roman domination number $4$ (see Figure \ref{fig:cp2c3c4}).

\begin{figure}[htp]
\begin{tikzpicture}
[inner sep=0.5mm, place/.style={circle,draw=black,fill=orange,thick},scale=0.7]

 [inner sep=0.5mm, place/.style={circle,draw=black,fill=orange,thick}]
\node[place] (v_1) at (-2,-1) [label=below:$1$] {};
\node[place] (v_2) at (0,-1) [label=below:$2$] {};
\node[place] (v_3) at (-2,1) [label=left:$-1$] {};
\node[place] (v_4) at (0,1) [label=right:$-1$] {};
\node[place] (v_5) at (-2,3) [label=above:$2$] {};
\node[place] (v_6) at (0,3) [label=above:$1$] {};
\tikzset{EdgeStyle/.append style = {blue, bend left=0}}
\Edge[](v_1)(v_4), \Edge[](v_1)(v_6), \Edge[](v_2)(v_3), \Edge[](v_2)(v_5), \Edge[](v_3)(v_6), \Edge[](v_4)(v_5)
[inner sep=0.5mm, place/.style={circle,draw=black,fill=orange,thick}]
\node[place] (v_11) at (2,-1) [label=below:$1$] {};
\node[place] (v_22) at (4,-1) [label=below:$2$] {};
\node[place] (v_33) at (2,1) [label=left:$-1$] {};
\node[place] (v_44) at (4,1) [label=right:$1$] {};
\node[place] (v_55) at (2,3) [label=above:$-1$] {};
\node[place] (v_66) at (4,3) [label=above:$-1$] {};
\node[place] (v_77) at (2,5) [label=above:$2$] {};
\node[place] (v_88) at (4,5) [label=above:$1$] {};
\tikzset{EdgeStyle/.append style = {blue, bend left = 50}}
\Edge[](v_11)(v_55), \Edge[](v_33)(v_77)
\tikzset{EdgeStyle/.append style = {blue, bend right = 50}}
\Edge[](v_22)(v_66), \Edge[](v_44)(v_88)
\tikzset{EdgeStyle/.append style = {blue, bend left=0}}
\Edge[](v_11)(v_44), \Edge[](v_11)(v_66), \Edge[](v_22)(v_33), \Edge[](v_22)(v_55), \Edge[](v_33)(v_66), \Edge[](v_44)(v_55),\Edge[](v_11)(v_88), \Edge[](v_33)(v_88),  \Edge[](v_55)(v_88),  \Edge[](v_22)(v_77), \Edge[](v_44)(v_77), \Edge[](v_66)(v_77)
\end{tikzpicture}
 \caption{SRDF for ${LC^C_3}$ and ${ LC^C_4}$.}
 \label{fig:cp2c3c4}
\end{figure}
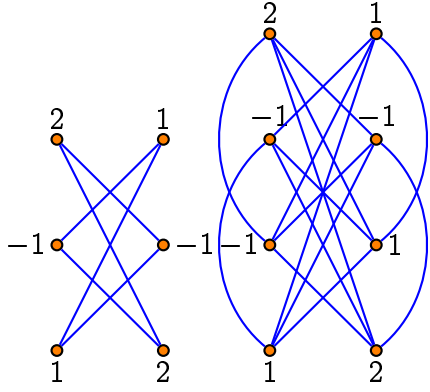
\end{example}

\begin{theorem}
Let $LC_n^C$ be the complement of the circular Ladder graph of order $n$. Then for all $n\geq 5$, we have
$$ \gamma_{SR} (LC_n^C)=     3. $$
\end{theorem}

\begin{proof}
For the complement of circular Ladder graph $LC_n^C$ when $n\geq $5, consider the function \[f:V(LC_n^C)\rightarrow \set{-1, 1, 2}\] given by:
$$f(v)=\begin{cases}-1,      &   \text{if } v=(1,1), v=(1,4), v=(2,i) \text{ for }~ 1\leq i\leq n, i\neq 2,5,\\
                                      1,      &    \text{if } v=(1,i) \text{ for } 5\leq i\leq n, v=(2,2),\\
                                      2,      &  \text{if } v=(1,2), v=(1,3) \text{ and } v=(2,5). 
           \end{cases} $$

Clearly, $f$ is a SRDF. By the construction of our labelling, it easy to show that 
\begin{equation*}
\begin{aligned}
|V_{-1}&|=2+(n-2)=n,\\
|V_1|&=(n-4)+1=n-3,\\
|V_2|&=3.
\end{aligned}
\end{equation*}
Hence, we have 
\begin{equation*}
\begin{split}
 \gamma_{SR} {(LC^C_n)}
&\leq 2|V_2|+|V_1|-|V_{-1}|=3.
\end{split}
\end{equation*}

Now, we prove our function is the minimum SRDF. Let us assume that \[g:V(LC^C_n)\rightarrow\set{-1, 1, 2}\] be another arbitrary SRDF on $LC^C_n,~ n\geq5$. Since the complement of the circular graph $LC^C_n$ is a regular graph of degree $2n-4$, so all vertices in $LC^C_n$ is adjacent to $2n-4$ vertices. To show our function is the minimum then we have the following cases for $g(v)$, depending on $f(v)=2~ \text{ or } f(v)=1$:
\begin{enumerate}
    \item If $g(V(LC^C_n))= f(V(LC^C_n))$ except $g(1,3)\neq f(1,3)$, since $f(1,3)=2$ but if $g(1,3)=1 < f(1,3)$. Clearly it is not possible to have $g(1,3)=1$, Otherwise we will have $\sum_{v\in N_{LC_n^C}[x]}g(v)=0$ where $x=(1,1)$ be the exterior vertex of $LC_n^c$, so
\begin{equation*}
\begin{split}
 \sum_{v\in N_{LC_n^C}[x]}g(v)
&= \sum_{i=1,i\neq 2}^{n-1}g((1,i)) + \sum_{i=2}^n g((2,i))\\
&=0.
\end{split}
\end{equation*}
\item If $g(V(LC_n^C))= f(V(LC_n^C))$ except $g(1,2)\neq f(1,2)$, since $f(1,2)=2$, but $g(1,2)=1 < f(1,2)$. Clearly, it is not possible to have $g(1,2)=1$ since otherwise we will have $\sum_{v\in N_{LC_n^C}[x]}g((1.i))=0$ for $i=2, 3, 4$ or $v=(2,3)$. 
\item $g(V(LC_n^c)) = f(V(LC_n^c)) $ except $g(2,5) \neq f(2,5)$ since $f(2,5)=2$, let $g(2,5)=1$. Clearly it is not possible to have $g(2,5)=1$ otherwise we will have $\sum_{v\in N_{LC_n^C}[x]}g(v)=0$ similar to the first case.
\item If $g(V(LC^C_n))= f(V(LC^C_n))$ except $g(2,2)=-1$, because $f(2,2)=1$. Clearly, it is not possible to have $g(2,2)=-1$, otherwise we will have 
\begin{equation*}
\begin{split}
 \sum_{v\in N_{LC_n^C}[x]}g(v)
&= \sum_{i=1,i\neq 2}^{n-1}g((1,i)) + \sum_{i=2}^n g((2,i))\\
&=-2 .
\end{split}
\end{equation*}
\end{enumerate}


This completes the proof. 
\end{proof}

\section{Concluding Remarks}\label{sec:conc}

\begin{enumerate}
    \item It is a natural next step to find the values of $\gamma_{SR}(G)$ when $G$ is either a grid graph or its complement. We leave that problem open.
    \item Some other closely related types of domination numbers have been studied in the literature. We point out two of these: the concept of signed Roman $k$-Domination in graphs by Henning and Volkmann \cite{HenningVolkmann} ($k=1$ corresponds to a SRDF), and the concept of weak signed Roman Domination in graphs by Volkmann \cite{Volk}. It would be interesting to study the graphs we study in this paper for these domination numbers.
\end{enumerate}

\bibliographystyle{alpha}
\bibliography{main}
\end{document}